\documentclass[]{article}

\usepackage{graphicx}
\usepackage{amsmath}
\usepackage{amsthm}
\usepackage{cite}

\newtheorem{thm}{Theorem}[section]
\newtheorem{lem}[thm]{Lemma}

\newcommand{\epn}{\ensuremath{\mathrm{epn}}}

\begin{document}
\title{A Lower bound for Secure Domination Number of an Outerplanar Graph}
\author{Toru Araki}
\date{Gunma University}
\maketitle

\begin{abstract}
  A subset $S$ of vertices in a graph $G$ is a secure dominating set
  of $G$ if $S$ is a dominating set of $G$ and, for each vertex
  $u \not\in S$, there is a vertex $v \in S$ such that $uv$ is an
  edge and $(S \setminus \{v\}) \cup \{u\}$ is also a dominating set
  of $G$.
  The secure domination number of $G$, denoted by $\gamma_{s}(G)$,
  is the cardinality of a smallest secure dominating sets of $G$.
  In this paper, we prove that for any outerplanar graph with
  $n \geq 4$ vertices, $\gamma_{s}(G) \geq (n+4)/5$ and the bound is
  tight.
  \vspace{1em}
  Secure dominating set,secure domination number, outerplanar graph,
  lower bound
\end{abstract}

\section{Introduction}
\label{sec:introduction}

We consider a finite undirected graph $G$ with vertex set $V(G)$ and
edge set $E(G)$ and use standard graph theoretic terminology and
notations (see, for example~\cite{chartrand11}).
The \emph{open neighborhood} of a vertex $v \in V(G)$ is defined by
$N_{G}(v) = \{u \mid vu \in E(G) \}$, and the \emph{closed
  neighborhood} of $v$ is $N_{G}[v]=N_{G}(u) \cup \{v\}$.
We denote by $\deg_{G}(v) = |N_{G}(v)|$ the degree of $v$.
For a subset $U \subseteq V(G)$, the subgraph induced by $U$ is
denoted by $G[U]$.
A vertex $v$ is a \emph{cut-vertex} of a connected graph $G$ if
$G-\{v\}$ is disconnected.
A \emph{block} of a connected graph of $G$ is a maximal subgraph
that has no cut-vertices.

A set $S \subseteq V(G)$ is a \emph{dominating set} of $G$ if each
vertex $u \in V(G) \setminus S$ is adjacent to some vertex in $S$.
The \emph{domination number} $\gamma(G)$ of $G$ is the smallest
cardinality of a dominating set of $G$.
Let $S$ be a dominating set of $G$.
A vertex $v \in S$ is said to \emph{defend} $u \not\in S$
if $uv \in E(G)$ and $S' =(S \setminus \{v\}) \cup \{u\}$ is also a
dominating set.
A dominating set $S$ is a \emph{secure dominating set} if, for any
$u \in V(G) \setminus S$, there exists a vertex $v \in S$ such that
$uv \in E(G)$ and $(S \setminus \{v\}) \cup \{u\}$ is also a
dominating set.
The \emph{secure domination number} $\gamma_{s}(G)$ of $G$ is the
smallest cardinality of a secure dominating set of $G$.

The problem of secure domination was introduced by Cockayne et al.~and
investigated some fundamental properties of a secure dominating set,
and obtained exact values of $\gamma_{s}(G)$ for some graph classes,
such as paths, cycles, complete multipartite
graphs~\cite{cockayne08:_protec}.
Various aspects of secure domination have been
researched~\cite{burger16,merouane15,burger08:_vertex,li17}.
Some polynomial-time algorithms for computing the secure domination
number of some restricted graph classes are investigated
~\cite{burger14,araki18:_secur,pradhan17,araki19:_secur,jha19,zou19,araki23,kisek21}.
Upper and lower bounds on $\gamma_{s}(G)$ have been established for
some graph classes~\cite{li17,merouane15,burger08:_vertex,araki18}.

A graph $G$ is \emph{outerplanar} if it has a crossing-free embedding
in the plane such that all vertices belong to the boundary of its
outer face (the unbounded face).
A \emph{maximal outerplanar graph} is an outerplanar graph such that
the addition of a single edge results in a graph that is not
outerplanar.
An inner face of a maximal outerplane graph $G$ is an \emph{internal
  triangle} if it is not adjacent to the outer face of $G$.
A maximal outerplane graph without internal triangles is called
\emph{stripped}.
Matheson and Tarjan~\cite{matheson96:_domin} proved a tight upper
bound for the domination number on the class of \emph{triangulated
  discs}: graphs that have an embedding in the plane such that all of
their faces are triangles, except possibly one.
They proved that $\gamma(G) \leq n/3$ for any $n$-vertex triangulated
disc.
Campos and Wakabayashi~\cite{campos13} showed that if $G$ is a maximal
outerplanar graph with $n$ vertices, then $\gamma(G) \leq (n+k)/4$ where
$k$ is the number of vertices of degree 2.
Tokunaga proved the same result independently
in~\cite{tokunaga13:_domin}.
Li et al.~\cite{li16} improved the result by showing that
$\gamma(G) \leq (n+t)/4$, where $t$ is the number of pairs of
consecutive degree 2 vertices with distance at least 3 on the outer
cycle.
For the secure domination problem in an outerplanar graph, the
author~\cite{araki18} proved that
$\gamma_{s}(G) \leq \lceil 3n/7 \rceil$ and the upper bound is sharp,
and also showed that for a stripped maximal outerplanar graph $G$,
$n/4 < \gamma_{s}(G) \leq \lceil n/3 \rceil$.

The purpose of this paper is to prove the next theorem that gives a
lower bound of the secure domination number for an outerplanar graph.

\begin{thm}
  \label{thm:lb}

  For an outerplanar graph with $n \geq 4$ vertices,
  \begin{displaymath}
    \gamma_{s}(G) \geq \frac{n+4}{5},
  \end{displaymath}
  and the bound is tight.
\end{thm}

Note that $K_{3}$ is an outerplanar graph and $\gamma_{s}(K_{3})=1$
which does not satisfy the lower bound.  So the condition $n \geq 4$
is necessary.


\section{Preliminaries}
\label{sec:prelim}

Let $S$ be a dominating set of $G$, and $u \in V(G) \setminus S$.
For a vertex $v \in S$, a vertex $u$ is an \emph{external private
  neighbor of $v$ with respect to $S$} if $N_{G}(u) \cap S = \{v\}$.
We denote by $\epn_{G}(v,S)$ the set of all external private neighbors
of $v$ with respect to $S$.
Cockayne et al.~\cite{cockayne08:_protec} proved a fundamental
property of a secure dominating set.

\begin{thm}[\cite{cockayne08:_protec}]
  \label{thm:sds_c}
  Let $S$ be a dominating set of $G$.
  A vertex $v \in S$ defends $u \in V(G) \setminus S$ if and only if
  $G[\epn_{G}(v,S) \cup \{u,v\}]$ is complete.
\end{thm}

A graph $H$ is a \emph{subdivision} of $G$ if either $G=H$ or $H$ can
be obtained from $G$ by inserting vertices of degree 2 into the edges
of $G$.
The following characterization of outerplanar graphs is known
(see~\cite{chartrand11}, for example).

\begin{thm}
  \label{thm:charact}
  A graph $G$ is outerplanar if and only if $G$ contains no subgraph
  that is a subdivision of the complete graph $K_{4}$ or the complete
  bipartite graph $K_{2,3}$.
\end{thm}

The following lemma is used in the next section.
A proof of Lemma~\ref{lem:bipatite_op} is given in Section~4
of~\cite{aita20}.

\begin{lem}
  \label{lem:bipatite_op}
  Let $G$ be a bipartite outerplanar graph with a bipartition $(X, Y)$
  and $|X| \geq 2$.
  If $\deg_{G}(v) \geq 2$ for any $v \in Y$, then $|Y| \leq 2|X|-2$.
\end{lem}


\section{Proof of Theorem~\ref{thm:lb}}
\label{sec:proof}

We can see that the theorem is true when $4 \leq n \leq 7$.
Hence we assume that $n \geq 8$.
Assume that $G$ is an outerplanar graph with $n$ vertices, and
$S$ is a secure dominating set of $G$ such that $\gamma_{s}(G) = |S|$.

Let $x=|S|$ and $y = |V(G) \setminus S|$.
If $y \leq 4x-4$, then we obtain $x \geq (n+4)/5$ since $n=x+y$.
To show $y \leq 4x-4$, we assume to the contrary that $y \geq 4x-3$.

Let $S=S_{2} \cup S_{1} \cup S_{0}$ be the partition of $S$, where
$S_{i} = \{v \mid |\epn_{G}(v, S)| = i\}$.
Note that, for any vertex $v \in S$, $|\epn_{G}(v, S)| \leq 2$ since
$G$ has no $K_{4}$ as a subgraph and Theorem~\ref{thm:sds_c}.
Further suppose that
\begin{displaymath}
  C = \{ u \mid u \in V(G) \setminus S \text{ and } |N_{G}(u) \cap S|
  \geq 2\}
\end{displaymath}
and $c=|C|$.
Thus $V(G) = S_{2} \cup S_{1} \cup S_{0} \cup \left( \bigcup_{v \in
    S_{2} \cup S_{1}}
  \epn_{G}(v, S)\right) \cup C$.
Hence $x = x_{2}+x_{1}+x_{0}$ and $y = (2x_{2}+x_{1})+c$, where $x_{i}
= |S_{i}|$ for $i = 0,1,2$.
Since $y \geq 4x-3$, we have $(2x_{2}+x_{1})+c \geq
4(x_{2}+x_{1}+x_{0})-3$, or
\begin{equation}
  \label{eq:3}
  c \geq 2x_{2} + 3x_{1} + 4x_{0} - 3.
\end{equation}

Let $B$ be a bipartite subgraph of $G$ defined as follows:
$B$ has the bipartition $(S, C)$, and $u \in S$ and $v \in C$ are
adjacent in $B$ if and only if $uv \in E(G)$.
Since $\deg_{B}(v) \geq 2$ for any $v \in C$, we have $c \leq 2x-2$ by
Lemma~\ref{lem:bipatite_op}.
Hence
\begin{equation}
  \label{eq:2}
  c \leq 2x_{2} + 2x_{1} + 2x_{0} - 2.
\end{equation}

From \eqref{eq:3} and \eqref{eq:2}, we obtain
\begin{equation}
  \label{eq:1}
  x_{1} + 2x_{0} \leq 1.
\end{equation}
Since $x_{1}$ and $x_{0}$ are integers, we have $x_{0}=0$ and
$x_{1} \leq 1$.
Each vertex in $C$ is defended by a vertex in $S$.
Since $G$ has no $K_{4}$ as a subgraph, no vertex in $S_{2}$ defends a
vertex in $C$ by Theorem~\ref{thm:sds_c}.
Also, since $G$ has no $K(2,3)$ as a subgraph, a vertex in $S_{1}$
defends at most two vertices in $C$ by Theorem~\ref{thm:sds_c}.
So $x=x_{2}+x_{1}$, $y = 2x_{2}+x_{1}+c$ and $c \leq 2x_{1}$.
\begin{itemize}
\item Assume that $x_{1}=0$.
  Then $x=x_{2}$ and $c=0$.
  By \eqref{eq:3}, we have $x_{2}=1$.
  Hence we obtain $x=1$ and $y=2$, and this contradicts $n \geq 8$.
\item Assume that $x_{1}=1$.
  Then $x=x_{2}+1$ and $c \leq 2$.
  By \eqref{eq:3}, we have $x_{2}=1$.
  Hence we obtain $x=2$ and $y \leq 5$, and this contradicts
  $n \geq 8$.
\end{itemize}
From the above discussion, we obtain a contradiction.
Hence we complete the proof of Theorem~\ref{thm:lb}.

\vspace{1em}

To show the lower bound is tight, we give an outerplanar graph $G$
with $n$ vertices such that $\gamma_{s}(G) = (n+4)/5$.
Let $n = 5k + 1$ for $k \geq 2$.
Thus $(n+4)/5 = k+1$.

For $k \geq 2$, a graph $G_{k}$ is defined as follows.
\begin{align*}
  V(G_{k}) &= \{v, v_{1}, v_{2}, \dots, v_{k}\} \cup \{w_{1}, w_{2},
             \dots, w_{2k}\} \cup \{u_{i}^{1}, u_{i}^{2} \mid i = 1,2,\dots, k\},
  \\
  E(G_{k}) &= \{vw_{i} \mid i = 1,2,\dots 2k\} \cup \{ v_{i} w_{2i-1},
             v_{i}w_{2i} \mid i = 1,2,\dots,k\} \\
           &{} \cup \{v_{i}u_{i}^{1},
             v_{i}u_{i}^{2}, u_{i}^{1}u_{i}^{2} \mid i = 1,2,\dots,k\}
\end{align*}
The graph $G_{k}$ is an outerplanar graph and has $5k+1$ vertices.
Figure~\ref{fig:extremal} illustrates $G_{k}$ when $k = 4$.
It is easy to see that the set $\{v, v_{1}, v_{2}, \dots, v_{k}\}$ of
$k+1$ vertices is a secure dominating set of $G$.

\begin{figure}[tbp]
  \centering
  \includegraphics{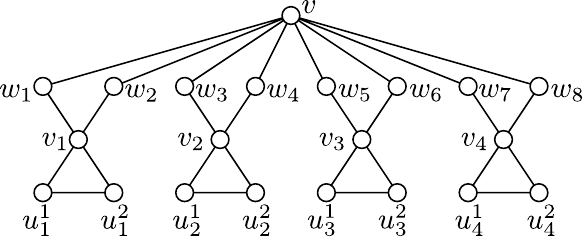}
  \caption{The outerplanar graph $G_{k}$ when $k=4$.}
  \label{fig:extremal}
\end{figure}

If an outerplanar graph $G$ with $5k+1$ vertices contains $G_{k}$ as a
spanning subgraph, then we obtain
\begin{displaymath}
  \frac{n+4}{5} \leq \gamma_{s}(G) \leq \gamma_{s}(G_{k}) = \frac{n+4}{5},
\end{displaymath}
and hence $\gamma_{s}(G) = k+1$.
Furthermore, we show that the converse is also true.

\begin{thm}
  \label{thm:extremal}
  An outerplanar graph $G$ with $n = 5k+1$ vertices for $k \geq 2$ has
  a secure dominating set of size $(n+4)/5$ if and only if $G$
  contains the graph $G_{k}$ as a spanning subgraph.
\end{thm}
\begin{proof}
  It is sufficient to show the necessity.
  Assume that $S$ is a minimum secure dominating set of $G$ and $|S| =
  k+1$.
  Then, we have $|V(G) \setminus S| = 4k$.
  Similar to the discussion in the proof Theorem~\ref{thm:lb}, we
  define $S_{2}, S_{1}, S_{0}$ and $C$, and $x_{i} = |S_{i}|$ and
  $c = |C|$.
  Since $|S| = k+1 = x_{2} + x_{1} + x_{0}$ and $|V(G) \setminus S| =
  4k = 2 x_{2} + x_{1} + c$, we obtain
  \begin{equation}
    \label{eq:4}
    c = 2 x_{2} + 3 x_{1} + 4 x_{0} - 4.
  \end{equation}
  Since $c \leq 2x-2$ by Lemma~\ref{lem:bipatite_op}, we have
  \begin{equation}
    \label{eq:5}
    c \leq 2 x_{2} + 2 x_{1} + 2 x_{0} - 2.
  \end{equation}
  From \eqref{eq:4} and \eqref{eq:5}, we obtain
  \begin{equation}
    \label{eq:6}
    x_{1} + 2 x_{0} \leq 2.
  \end{equation}

  \vspace{1em}
  \noindent
  \textbf{Claim 1.} $x_{1} = 0$ and $x_{0} = 1$.

  \noindent
  \textit{Proof of Claim~1.}
  Assume to the contrary that $x_{1} \geq 1$.
  By \eqref{eq:6}, $x_{0} = 0$ and $x_{1} = 1$ or 2.
  As mentioned earlier, any vertex in $S_{2}$ defends at most two
  vertices, and any vertex in $S_{1}$ defends at most three vertices.
  Hence, it satisfies $2 x_{2} + 3 x_{1} \geq |V(G) \setminus S| =
  4k$.
  If $x_{1} = 1$ and $x_{2} = k$, we have $2k + 3 \geq 4k$ or
  $k \leq 3/2$, which is a contradiction.
  If $x_{1} = 2$ and $x_{2} = k-1$, we have $2(k-1) + 2 \geq 4k$ or
  $k \leq 0$, which is also a contradiction.

  Hence we obtain $x_{1} = 0$ and $x_{0} \leq 1$.
  Assume that $x_{0} = 0$.
  Then $S = S_{2}$ and the vertices in $S$ defend at most $2(k+1) <
  4k$ vertices, and hence a contradiction occurs.

  \vspace{1em}
  By Claim~1, we obtain $x_{2} = k$, $x_{1} = 0$, and $x_{0} = 1$.
  Let $S_{2} = \{v_{1}, v_{2}, \dots, v_{k}\}$ and $S_{0} = \{v\}$, and
  $\epn_{G}(v_{i}, S) = \{u_{i}^{1}, u_{i}^{2}\}$ for every
  $v_{i} \in S_{2}$.

  \vspace{1em}
  \noindent
  \textbf{Claim 2.}  For each $1 \leq i \leq k$, three vertices
  $v_{i}, u_{i}^{1}, u_{i}^{2}$ form a triangle.

  \noindent
  \textit{Proof of Claim~2.}
  Since $S$ is a secure dominating set and $u_{i}^{1}, u_{i}^{2}$ are
  the external private neighbors of $v_{i}$, the three vertices
  $v_{i}$, $u_{i}^{1}$ and $u_{i}^{2}$ from a triangle by
  Theorem~\ref{thm:sds_c}.

  \vspace{1em}

  Since $|V(G) \setminus S| = 4k$ and
  $V(G) \setminus S = \{u_{1}^{1}, u_{1}^{2}, \dots,
  u_{k}^{1},u_{k}^{2}\} \cup C$, we have $|C| = 2k$.  Let
  $C = \{w_{1}, w_{2}, \dots, w_{2k}\}$.

  \vspace{1em}
  \noindent
  \textbf{Claim 3.}
  Each vertex in $C$ is adjacent to $v \in S_{0}$.

  \noindent
  \textit{Proof of Claim~3.}
  Since any $v_{i} \in S_{2}$ defends the two vertices $u_{i}^{1}$ and
  $u_{i}^{2}$ but cannot defend any vertex in $C$, vertex
  $w_{i} \in C$ is defended by $v \in S_{0}$.
  Thus $w_{i}$ is adjacent to $v$ for $i=1,2,\dots,2k$.

  \vspace{1em}
  \noindent
  \textbf{Claim 4.} For any $v_{i}, v_{j} \in S_{2}$ with $i \neq j$,
  $v_{i}$ is adjacent to exactly two vertices in $C$ and
  $N_{G}(v_{i}) \cap N_{G}(v_{j}) = \emptyset$.

  \noindent
  \textit{Proof of Claim~4.}
  Assume that some $v_{i} \in S_{2}$ is adjacent to three vertices
  $w_{p}, w_{q}, w_{r} \in C$.
  Since $v$ is adjacent to the three vertices by Claim~3, the induced
  subgraph by the set $\{v, v_{i}, w_{p}, w_{q}, w_{r}\}$ contains
  $K_{2,3}$ as a subgraph, and thus a contradiction occurs by
  Theorem~\ref{thm:charact}.
  Hence, each $v_{i} \in S_{2}$ is adjacent to at most two vertices
  in $C$.
  Since $|S_{2}| = k$ and $|C|=2k$, for $i \neq j$, $v_{i}$ is
  adjacent to exactly two vertices in $C$ and $N_{G}(v_{i}) \cap
  N_{G}(v_{j}) = \emptyset$.

  \vspace{1em}
  \noindent
  By Claim~2, 3, and~4, $G$ contains the graph $G_{k}$ as a spanning
  subgraph.
\end{proof}


\end{document}